%% file: monoidal.tex
\documentclass[10pt,a4paper]{article}
\usepackage[utf8]{inputenc}
\usepackage[T1]{fontenc}
\usepackage{lineno,hyperref}
\usepackage{amssymb}
\usepackage{amssymb}
\usepackage{amsthm}
\usepackage{mathtools}
\usepackage{tikz}
\usepackage{tikz-cd}
\usepackage{graphicx}
\usepackage{float}

\theoremstyle{definition} 
\newtheorem{defi}{Definition}

\theoremstyle{remark} 
\newtheorem{rem}{Remark}
\newtheorem{cro}{Corollary}
\newtheorem*{notation}{Notation}
\newtheorem*{discussion}{Discussion}
\newtheorem*{outline}{Outline}
\newtheorem*{convention}{Convention}

\newcommand{\obj}{Obj}
\author{Fatimah Rita Ahmadi \\
	\multicolumn{1}{p{.7\textwidth}}{\centering\emph{Department of Computer Science\\University of Oxford}}}
\title{Monoidal 2-Categories: A Review}
\begin{document}
\maketitle
\begin{abstract}
We review the complete definition of monoidal 2-categories and recover Kapranov and Voevodsky’s definition from the algebraic definition of weak 3-category(or tricategory). 
\end{abstract}
\section{Introduction}
Kapranov and Voevodsky, in their seminal paper \cite{KV}, introduced monoidal 2-categories. Their definition comprises a list of axioms, the categorical origin of which does not seem transparent.

A monoidal n-category, by the categorical procedure, can be obtained by taking a one-object weak n + 1-category. For instance, a monoidal category is defined as one object weak 2-category(or bicategory). Due to the lack of a complete definition for weak n-categories, $n$ cannot be any arbitrary natural number. However, thanks to the algebraic definition of weak 3-category(or tricategory) introduced by Gordon, Power and Street \cite{Gordon} and Gurski \cite{Gurski}, one can define a monoidal 2-category as a tricategory with one object.

After a full list of necessary data, Gurski’s thesis gives two main axioms, namely, Stasheﬀ and unit polytopes: the 2-dimensional correspondence of pentagonal and triangle equations for monoidal categories. However, it does not explicitly spell out all diagrams obtainable by naturality, modification, and (2-)functoriality conditions. Unpacking these conditions, one can recover KV’s axioms, which underlines the goal of this paper.

The recovery procedure needs to consider two main issues. The first one is the difference between tensorators in KV’s version and our version. Kapranov and Voevodsky defined three tensorators which can be obtained by a single tensorator of the tricategory approach. The second difference is an extra piece of data given by KV which can be written based on other data: a 2-morphism $\epsilon$ between left and right 1-unitors indexed by the unit object, i.e., $r_I, l_I: I \otimes  I \longrightarrow I$. To show it is redundant, we present a proof sketch in the comparison section \ref{sec:extra_data}.

\begin{discussion}\label{discussion}
Baez and Neuchl \cite{baez} reviewed the semi-strict definition of monoidal 2-category. Stay, on the other hand, spelt out the definition but without tensorators \cite{Stay}. Schommer Pries cited Stay \cite{schom}. The combination of both constructs the full description of a monoidal 2-category.  

Moreover, Stay listed four unit polytopes by alternating the location of the unit object, but only two of them are necessary, namely when the unit object is the second or third object in the polytope. As Gurski proved the other two are corollary of these axioms. Note also Stay's diagrams, in our case, need to be revised, as in the presence of tensorators, filling 2-morphisms will be modified based on the modified tensor product. 
\end{discussion}
\begin{rem}
Higher categories and particularity 2-categories have gained attention in topological quantum field theory and topological condensed matter physics, for instance, Kitaev and Kong modelled gapped boundaries in Levin-Wen model by the bicategory of module categories over fusion category of the bulk \cite{Kitaev}. The current review tries to be accessible for physicists. Hence, some readers will encounter some parts which may seem excessive, for example, items 3 and 4 in Definition \ref{sec:monoidal-data}. 
\end{rem}
\begin{rem}
For a concise review of bicategories consult Leinster's paper \cite{Leinster}. By a 2-category, we mean a strict 2-category. In other words, a bicategory whose associators and unitors are identities.
\end{rem}
\begin{outline}
The paper is organized in the following way: having in mind both
KV’s and Gurski’s definitions, Section \ref{sec:def-monoidal} presents the definition of monoidal 2-category. So you will see the same structures and diagrams as KV whose origins are explained in Gurski’s tricategory language. We ﬁrst list the data \ref{sec:monoidal-data}; then in Subsection \ref{sec:conditions}, we give all conditions under three different groups: Naturality Conditions \ref{sec:natural}, Modification Conditions \ref{sec:modification} and Axioms \ref{sec:axioms}. The diagrams of these subsections are depicted in Appendix \ref{sec:figures}. The paper concludes with Section \ref{sec:compar} in which we compare the produced definition with the original definition given by Kapranov and Voevodsky. Appendices present a minimal account of the definitions of 2-functors Appendix \ref{sec:2-functor} modification and an explicit procedure for obtaining modification diagrams for pentagonator and 2-unitors Appendix \ref{sec:modification-data}.	
\end{outline}
\section{Acknowledgments}
The author thanks Steve Simon for his encouragement to prepare this material as a review paper. 
\section{Definition of Monoidal 2-Categories} \label{sec:def-monoidal}
\begin{notation}
Capital letters, $A, B, ..$ are reserved for objects, small letters $f, g, ...$ for 1-morphisms and Greek letters $\alpha, \beta, ...$ for 2-morphisms. Horizontal composition of 1- and 2-morphisms is denoted by juxtaposition and vertical composition of 2-morphisms by $\odot$. For every object $A$, the identity 1-morphism or 1-identity is shown by $id_A$, but in the tensor product of $id_A$ with 1-morphisms $id_A \otimes f$, we leave out id and denote it by $A \otimes f$. For every 1-morphism $f$, the identity 2-morphism or 2-identity is represented by $1_f$ . Whenever, it is clear from the context, in the horizontal composition of 1-identities with 2-morphisms, we leave out 1 and use whiskering convention $1_f \circ \alpha = f\alpha$. 
\end{notation}

\begin{convention}
Following the style of the stunning figures of Stay’s paper, we color diagrams. Shapes ﬁlled by pink are penetrators $\pi$, those by blue are tensorators $\phi$, and those ﬁlled by brown are 2-unitors. Unfilled ones are natural 2-isomorphisms.	
\end{convention}

\begin{defi}
A monoidal 2-category is a 2-category $\mathcal{C}$ with a 2-functor $\otimes: \mathcal{C} \times \mathcal{C} \longrightarrow \mathcal{C}$  and a list of data given in Section \ref{sec:monoidal-data} subject to some conditions presented in Section \ref{sec:conditions}.
 \subsection{Data}\label{sec:monoidal-data}
	\begin{enumerate}
	\item A \textit{unit object} $I \in \obj(\mathcal{C})$.
	\item For every pair of objects $A$, $B$ in $\obj(\mathcal{C})$, the \textit{tensor product of objects} is an object $A \otimes B \in \obj( \mathcal{C})$ denoted by juxtaposition $AB$. 
	\item For 1-morphisms $A \xrightarrow{f} B$ and  $C \xrightarrow{g} D$, the \textit{tensor product of 1-morphisms} defined as $A \otimes C \xrightarrow{f \otimes g} B \otimes D$ exists. 
	\item For every pair of 2-morphisms, $\alpha: f\Rightarrow h$, $\beta: g \Rightarrow k$,  and 1-morphisms $f, h:A \longrightarrow B$ and $g, k: C \longrightarrow D$, there exists a \textit{tensor product of  2-morphism} $\alpha \otimes \beta: f \otimes g \Rightarrow h \otimes k$ such that $f\otimes g , h \otimes k \in Hom(A\otimes C, B \otimes D)$. 
	\item Functoriality of 2-functor $\otimes$ holds up to natural 2-isomorphisms, meaning, it preserves composition of 1-morphisms and identity 1-morphisms up to natural 2-isomorphisms Appendix \ref{sec:2-functor}.
	\begin{itemize}
		\item 	For every quadruple of 1-morphisms $(f, f', g, g')$, such that $f: A \longrightarrow A^\prime,f^\prime: A^{\prime} \longrightarrow A^{\prime \prime}, g: B \longrightarrow B^\prime, g^\prime: B^\prime \longrightarrow B^{\prime \prime}$, a natural 2-isomorphism called \textit{tensorator}:
		\begin{equation}\label{tensorator-1}
			\phi_{f^\prime, g^\prime, f, g}: (f^\prime \otimes g^\prime) \circ (f \otimes g) \Rightarrow (f^\prime \circ  f) \otimes (g^\prime \circ g)
		\end{equation}
In addition to naturality condition, it has to satisfy Equations \ref{tensorator:axiom1}. Note that KV’s definition introduces three different tensorators which the tricategory approach packs in one single tensorator denoted by $\phi$, for further discussion check Section \ref{sec:tensorators}.
		\item
		For every pair of objects $A, B$, we also need a natural 2-isomorphism:
		\begin{equation}
			\phi_{A, B} : id_A \otimes id_B \Rightarrow id_{A \otimes B}
		\end{equation}
For our purpose of recovering KV’s definition, we let $\phi_{A,B}$ be identity 2-morphisms. This assumption has a consequence which we elaborate on in Section \ref{sec:consequence}. 
\end{itemize}
\item For every triplet of objects $A, B, C$, a natural isomorphism called \textit{1-associator}:
$$a: \otimes (\otimes \times \mathbf{Id}) \longrightarrow \otimes (\mathbf{Id} \times \otimes )$$
Since $a$ is between two 2-functors, it consists of two pieces of data: a 1-isomorphism indexed by three objects mentioned earlier and a natural 2-isomorphism indexed by at least one 1-morphism. For instance, for $f: A \longrightarrow A^\prime$ there exists a natural 2-isomorphism $\alpha_{f, B, C}$ subject to the naturality conditions in the 1-morphism $f$, objects $B$ and $C$, and also Axiom \ref{eq:transformation1}.
\begin{equation}
	\begin{tikzpicture}
		\node (0) at (0, 0) {$(A B)C$};
		\node (1) at (3, 0) {$A (BC)$};
		\node (2) at (0, 2){$(A^\prime B)C$};
		\node (3) at (3, 2) {$A^\prime(BC)$};
		\node[rotate=150] (4) at (1.5, 1) {$\Rightarrow$};
		\node (5) at (1.5, 1.5){$\alpha_{f, B, C}$};
		\draw[->] (0) to node[below]{$a_{A, B, C}$} (1);
		\draw[->] (0) to node[left]{$(f\otimes B)\otimes C$} (2);
		\draw[->] (1) to node[right]{$f\otimes (B\otimes C)$} (3);
		\draw[->] (2) to node[above]{$a_{A^\prime, B, C}$} (3);
	\end{tikzpicture}
\end{equation}
\item For every object $A$, a natural isomorphism called \textit{left 1-unitor},
$$l_A: I \otimes -  \longrightarrow \mathbf{Id} $$
Similar to 1-associators, because $l$ is a natural 2-transformation between two 2-functors $I \otimes -$ and $\mathbf{Id}$, there should be a natural 2-isomorphism indexed by a 1-morphism, $l_f$, which further satisfies the naturality condition and Axiom \ref{eq:transformation1}. 
\begin{equation}
	\begin{tikzpicture}
		\node (0) at (0, 0) {$IA$};
		\node (1) at (3, 0) {$IA^\prime $};
		\node (2) at (0, 2){$A$};
		\node (3) at (3, 2) {$A^\prime $};
		\node[rotate=150] (4) at (1.5, 1) {$\Rightarrow$};
		\node (5) at (1.5, 1.5){$l_f$};
		\draw[->] (0) to node[below]{$I \otimes f$} (1);
		\draw[->] (0) to node[left]{$l_A$} (2);
		\draw[->] (1) to node[right]{$l_{A^\prime }$} (3);
		\draw[->] (2) to node[above]{$f $} (3);
	\end{tikzpicture}
\end{equation}
\item For every object $A$, a natural isomorphism called \textit{right 1-unitor}:
$$r: - \otimes I \longrightarrow \mathbf{Id}$$
This includes a natural 1-isomorphism $r_A: A \otimes I \longrightarrow A$, and a natural 2-isomorphism $r_f$ subject to the naturality conditions and Axiom \ref{eq:transformation1}. 
\begin{equation}
	\begin{tikzpicture}
		\node (0) at (0, 0) {$AI$};
		\node (1) at (3, 0) {$A^\prime I$};
		\node (2) at (0, 2){$A$};
		\node (3) at (3, 2) {$A^\prime $};
		\node[rotate=150] (4) at (1.5, 1) {$\Rightarrow$};
		\node (5) at (1.5, 1.5){$r_f$};
		\draw[->] (0) to node[below]{$f \otimes I$} (1);
		\draw[->] (0) to node[left]{$r_A$} (2);
		\draw[->] (1) to node[right]{$r_{A^\prime }$} (3);
		\draw[->] (2) to node[above]{$f $} (3);
	\end{tikzpicture}
\end{equation}
		\item For every four objects $A, B, C, D$, there is a modification between composition of 1-associators called pentagonator $\pi_{A, B, C, D}$ shown in Figure \ref{fig:Pink}. Note that $\pi$ is not natural but it should satisfies the modification condition Equation \ref{modification-square}.
\begin{equation}\label{fig:Pink}
	\begin{tikzpicture}
		\filldraw[white,fill=red,fill opacity=0.1](0,3)--(2,4)--(4,2)--(2,0)--(0,1)--cycle;
		\node (1) at (0,3) {$((A B) C) D$};
		\node (2) at (2,4) {$(A B) (C D)$};
		\node (3) at (4,2) {$A (B (C D))$};
		\node (4) at (2,0) {$A ((B C) D)$};
		\node (5) at (0,1) {$(A (B C)) D$};
		\node[rotate=90] (6) at (1,2) {$\Rightarrow$};
		\node (7) at (1.9,2) {$\pi_{A, B, C, D}$};
		
\draw[->] (1) to node[above left] {$a$} (2);
\draw[->] (2) to node [above right] {$a$} (3);
\draw[->] (5) to node [below left] {$ a$} (4);
\draw[->] (4) to node [below right] {$A \otimes a$} (3);
\draw[->] (1) to node [left] {$a \otimes D$} (5);
	\end{tikzpicture}
\end{equation}	
\item For every two objects $A, B$, there exist three 2-isomorphisms called 2-unitors Equation \ref{fig:2-unitors}. They are modifications, hence, subject to the modification condition \ref{modification-square}.
\begin{equation}\label{fig:2-unitors}
	\begin{tikzpicture}
		\filldraw[white,fill=brown,fill opacity=0.5](0,2)--(3,1)--(0,0)--cycle;
		\node (0) at (0,2) {$(I A) B$};
		\node (1) at (3,1) {$I (A  B)$};
		\node (2) at (0,0) {$A  B$};
		\node (3) at (1,1) {$\Rightarrow \lambda$};
		\draw[->] (1)  to node [below right] {$l$} (2);
		\draw [->] (0) to node [above right] {$a$} (1);
		\draw[->] (0) to node [below, sloped] {$l \otimes B$} (2);
	\end{tikzpicture}
	\begin{tikzpicture}
		\filldraw[white,fill=brown,fill opacity=0.5](0,2)--(3,1)--(0,0)--cycle;
		\node (0) at (0,0) {$A  B$};
		\node (1) at (3,1) {$A (I B)$};
		\node (2) at (0,2) {$(A I)  B$};
		\node at (1,1) {$\Rightarrow \mu$};
		\draw [->] (1) to node [ below, sloped] {$A \otimes  l$} (0);
		\draw [->] (2) to node [above] {$a$} (1);
		\draw [->] (2) to node [below, sloped] {$r \otimes B$} (0);		
	\end{tikzpicture}
	\begin{tikzpicture}
		\filldraw[white,fill=brown,fill opacity=0.5](0,2)--(3,1)--(0,0)--cycle;
		\node (0) at (0,2) {$(A  B)  I$};
		\node (1) at (3,1) {$A  (B  I)$};
		\draw [->] (0) to node [above right] {$a$} (1);
		\node (2) at (0,0) {$A  B$};
		\draw [->] (2) to node [below, sloped] {$A \otimes r$} (1);
		\draw [->] (0) to node [left] {$r$} (2);
		\node at (1,1) {$\Rightarrow \rho$};
	\end{tikzpicture}
\end{equation}
\end{enumerate}
\begin{rem}\label{modified-tensor}
Before listing the conditions on data, we should notice that the existence of tensorators changes the tensor product of an object with a 2-morphism. We denote the new tensor product with $\hat{\otimes}$. We show a more comprehensible example ﬁrst, then restate KV's example Figure \ref{fig:kv-exp} for $\pi_{A, B, C, D} \hat{\otimes} E$. 
\begin{itemize}
\item 
Assume that there exists a 2-morphism $\alpha: gf \Rightarrow h$,  now if we tensor each object from left with an object $A$, the filling 2-morphism will not be modified as
$A \otimes \alpha$, but it will be $(A \otimes \alpha) \circ \phi_{A,g,A,f}$.
\begin{center}
\begin{tikzpicture}
\node (0) at (0, 0) {$A \otimes B$};
\node (1) at (3, 0) {$A \otimes B^\prime$};
\node (2) at (6, 0) {$A \otimes B^{\prime \prime}$};
\node[rotate=270] at (3, -0.5) {$\Rightarrow$};
\node  at (3.8, -0.5) {$\phi_{A, g, A, f}$};
\node[rotate=270] at (3, -1.7) {$\Rightarrow$};
\node  at (3.8, -1.7) {$A \otimes \alpha$};

\draw[->] (0) to node[above]{$A \otimes f$} (1);
\draw[->] (1) to node[above]{$A \otimes g$} (2);
\draw[->, bend right= 30] (0) to node[below]{$A \otimes gf$} (2);
\draw[->, bend right= 80] (0) to node[below]{$A \otimes h$} (2);
\end{tikzpicture}
\end{center}
\item Now take the example presented by Kapranov and Voevodsky Figure \ref{fig:kv-exp}: consider the pentagonator if each object is tensored by an object $E$ from right. One should observe that the filling 2-morphism is $\pi_{A, B, C, D} \hat{\otimes} E$ which is different from $\gamma = \phi \otimes E$. 
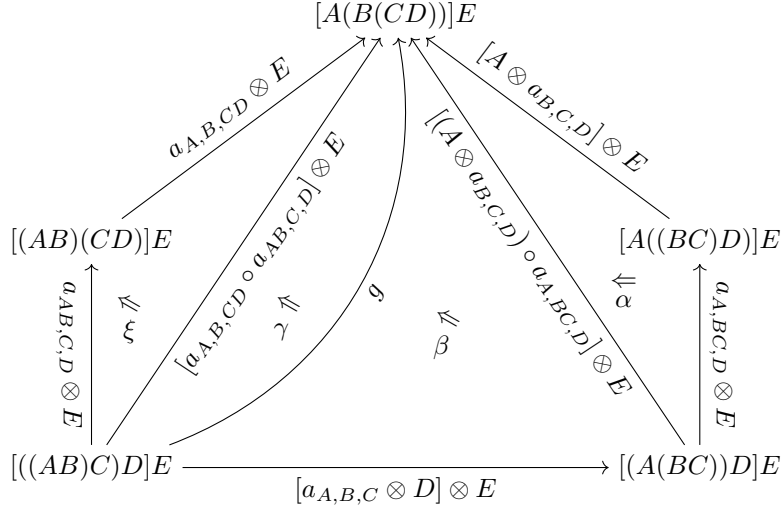
\begin{figure}[h]
\centering
\begin{tikzpicture}
		\node (1) at (0, 0) {$[((AB)C)D]E$};
		\node (2) at (8, 0) {$[(A(BC))D]E$};
		\node (3) at (0, 3) {$[(AB)(CD)]E$};
		\node (4) at (8, 3) {$[A((BC)D)]E$};
		\node (5) at (4, 6) {$[A(B(CD))]E$};
		\node[rotate=180] (6) at (7, 2.5){$\Rightarrow$};
		\node (7) at (7, 2.2){$\alpha$};
		\node[rotate=150] (8) at (4.7, 2){$\Rightarrow$};
		\node (9) at (4.6, 1.6){$\beta$};
		\node[rotate=150] (10) at (2.6, 2.2){$\Rightarrow$};
		\node (11) at (2.5, 1.8){$\gamma$};
		\node[rotate=150] (12) at (0.5, 2.2){$\Rightarrow$};
		\node (13) at (0.5, 1.8){$\xi$};
		
		\draw[<-] (2) to node[below, sloped]{$[a_{A, B, C} \otimes D]\otimes E$} (1);
		\draw[<-] (3) to node[below, sloped]{$a_{AB, C, D} \otimes E$}(1);
		\draw[<-] (5) to node[above, sloped]{$a_{A, B, CD}\otimes E$}(3);
		\draw[<-] (5) to node[above, sloped]{$[A \otimes a_{B, C, D}] \otimes E$}(4);
		\draw[<-] (4) to node[above, sloped]{$a_{A, BC, D} \otimes E$}(2);
		\draw[<-] (5) to node[below, sloped]{$[(A \otimes a_{B,C,D}) \circ a_{A,BC,D}] \otimes E$}(2);
		\draw[<-, bend left=40] (5) to node[below, sloped]{$g$}(1);
		\draw[<-] (5) to node[below, sloped]{$[a_{A,B,CD}\circ a_{AB,C,D} ] \otimes E$}(1);
	\end{tikzpicture}
	\caption{~The missing 1-morphisms in the figure is
	$g = [(A \otimes a_{B,C,D}) \circ a_{A,BC,D} \circ (a_{A,B,C} \otimes D)] \otimes E$, and the missing 2-morphisms are $\alpha = \phi_{A\otimes a_{B,C,D},a_{A,BC,D},E}$, $\gamma = \pi \otimes E$, $\beta = \phi_{[(A \otimes a_{B, C, D}) \circ a_{A, BC, D}], E, [a_{A, B, C} \otimes D], E}$, $\xi = \phi^{-1}_{a_{A,B,CD},E,a_{AB,C,D},E}$.}\label{fig:kv-exp}
\end{figure}
\end{itemize}
\end{rem}
\end{defi}
\begin{rem}
	As mentioned briefly in the discussion, Stay did not use the modified tensor product. Since the paper assumes tensorators are identity 2-morphisms. However, in the following, we shall work with the modified tensor product when it is necessary.
\end{rem}
\subsection{Condition on Data}\label{sec:conditions}
Each of the data above should further satisfy some conditions either due to naturality or axioms of 2-transformation or modification. However, similar to monoidal categories, in addition to the conditions arising from the nature of the data, monoidal 2-categories are further subject to two axioms called Stasheﬀ and unit polytopes. Due to the coherence theorems, these two conditions are enough and no more conditions are required.
\subsubsection{Naturality Conditions}\label{sec:natural}
\begin{enumerate}
\item \textbf{Naturality of 1-associators}: weakening 2-isomorphisms $\alpha_{f,B,C}$ for 1-associators should be natural in its indices $f, B, C$.
\begin{itemize}
	\item  \underline{Naturality in 1-Morphism $f$}: if $\gamma :f \Rightarrow f^\prime$, the cylinder shown in Equation \ref{fig:nat-ass1} commutes. Similar cylinders when 1-morphism is the ﬁrst or second index should commute.
	\item \underline{Naturality in Object $B$}: to check the naturality in $B$, assume another object $B^\prime$, 1-morphisms $f: A \longrightarrow A^\prime$ and $g: B \longrightarrow B^\prime$, then, the naturality squares of $\alpha$ result in the white squares depicted in the cube \ref{fig:nat-1-ass2}. Gluing squares results in the cube.
	\item \underline{Axiom \ref{eq:transformation1} of naturality}: for any composable pair of 1-morphisms $A \xrightarrow{f} A^\prime \xrightarrow{f^\prime} A^{\prime \prime}$, Diagram \ref{fig:axiom-ass} commutes.
\end{itemize}
\item \textbf{Naturality of 1-unitors}: weakening 1-unitors $l_f$ and $r_f$ need to satisfy naturality conditions and Axiom \ref{eq:transformation1}. We only present the conditions for the left unitors $l_f$ . Corresponding conditions should hold for the right unitors $r_f$ .
\begin{itemize}
\item \underline{Naturality in 1-Morphism $f$}: 
the left and right 1-unitors are natural, and their natural 2-isomorphisms $l_f$ and $r_f$ are indexed by a 1-morphism.  Hence, to check naturality of $l_f$ and $r_f$ we alternate $f$; that is, for $\gamma: f \Rightarrow f^\prime$, the cylinder shown in Figure \ref{fig:nat-unit1} commutes. 
\item \underline{Axiom \ref{eq:transformation1} of naturality:} For any composable pair of 1-morphisms $A \xrightarrow{f} A^\prime \xrightarrow{f^{\prime}} A''$, the triangle prism shown in Figure \ref{fig:nat-unit2} commutes. 
\end{itemize}
\item \textbf{Naturality of tensorators:} Observing and finding appropriate diagrams are easier if one tries to alternate indices of KV's tensorators Section \ref{sec:tensorators}. We list three classes of them and explicitly mention how you can find these diagrams. One might prefer to check Section \ref{sec:tensorators} before proceeding further. 
\begin{itemize}
\item The first condition is obtained by checking the naturality of tensorator $\otimes_{f,g}$ in $f$ and $g$ (this is the symbol that KV used). If one alternates $f$
or $g$ in $\otimes_{f,g}$ or equally in $\phi^{-1}_{A^\prime, g, f, B} \odot \phi_{f, B^\prime, A, g}$, for $\gamma : f \longrightarrow f^\prime$, the cylinder in Figure \ref{fig:nat-tensor1} should commute.
\item If one alternates one of the 1-morphisms of $\otimes_{g, f, B}$ or in $\phi_{g, B, f, B}$, then for every $\alpha$ and $g$, 
\begin{equation}
	\begin{tikzpicture}
		\node (1) at (0, 0) {$A$};
		\node (2) at (2, 0) {$A^\prime$};
		\node (3) at (4, 0) {$A^{''}$};
		\node (4) at (1, 0.5){};
		\node (5) at (1, -0.5){};
		\draw[->, bend right = 40] (1) to node[below]{$f^\prime$}(2);
		\draw[->, bend left=40] (1) to node[above]{$f$} (2);
		\draw[->] (2) to node[above]{$g$} (3);
		\draw[->, double] (4) to node[right]{$\alpha$} (5);
	\end{tikzpicture}
\end{equation}
Figure \ref{fig:nat-tensor2} commutes. 
\item Now if we check the naturality of $\otimes_{f^\prime , f, B}$ in Object $B$, for  $A \xrightarrow{f} A^\prime \xrightarrow{f^\prime } A^{''}$ and $g: B \longrightarrow B^\prime$, the triangle prism commutes Figure \ref{fig:nat-tensor3}.
\item \underline{Axiom \ref{eq:transformation1} for tensorators:} For any composable triplet of 1-morphisms $A \xrightarrow{f} A^\prime \xrightarrow{f^\prime} A^{''} \xrightarrow{f^{''}} A^{'''}$ and an object $B$,  Tetrahedron \ref{fig:axiom-tensor} commutes. 
\end{itemize}
\end{enumerate}
\subsubsection{Modification Data}\label{sec:modification}
Pentagonator and 2-unitors are modification. Hence, the axiom of modification \ref{modification-square} should hold for them. We present the details of obtaining the below diagrams from the modification square in Appendix \ref{sec:modification-data}.
\begin{itemize} 
\item \textbf{2-Unitors:} The modification square for every $f: A \longrightarrow A^\prime$  results in a two dimensional diagram. Gluing the boundaries of the 2-dimensional picture gives us the prism of Figure \ref{fig:modi-unit1} .
\item \textbf{Pentagonator:} for every 1-morphism $f: A \longrightarrow A^\prime $ and objects $B, C, D$ the pentagonal prism \ref{fig:modi-penta} commutes. 
\end{itemize} 
\subsubsection{Axioms}\label{sec:axioms}
\begin{itemize}
\item For every five objects, $(A, B, C, D, E)$,  Stasheff polytopes Figure \ref{fig:stashef} should commute.
\item For every triplet of objects $A, B, C$, the unit polytopes figures \ref{fig:unit-poly1} and \ref{fig:unit-poly2} commute.
\end{itemize}
\begin{defi}
A monoidal 2-category is \textbf{strict} if all weakening 1- and 2-morphisms are identities. It is \textbf{semi-strict} if all weakening 1- and 2-morphisms except tensorators $\phi$ are identities.
\end{defi}
\section{Comparison Between Definitions}\label{sec:compar}
\subsection{Isomorphism or (adjoint)equivalence}
Gordon, Power and Street\cite{Gordon} defined tricategories whose structure 1-morphisms are all equivalences. However, Gurski replaced them with adjoint equivalence. To recover KV’s definition, we make them even stronger, and substitute them
with isomorphisms.
\subsection{An extra data}\label{sec:extra_data}
KV’s definition included an extra piece of data which is unnecessary and can be obtained by 2-morphisms listed above. They defined a special 2-morphism between the left and right unitors for the unit object,
\begin{center}
\begin{tikzpicture}
	\node (0) at (0, 0) {$I \otimes I$};
	\node (1) at (3, 0) {$I$};
	\node[rotate = 270] (2) at (1.5, 0) {$\Rightarrow$};
	\node (3) at (2, 0) {$\epsilon$};
	\draw[->, bend left = 30] (0) to node[above]{$l_I$} (1); 
	\draw[->, bend right = 30] (0) to node[below]{$r_I$} (1); 
\end{tikzpicture}
\end{center}
such that it satisfies the diagram below, 
	\begin{figure}[H]
	\input{images/compat4.tikz}
	\label{Pastings}
\end{figure}
Although tiresome, one can write $\epsilon$ based on other 2-morphisms by taking the figure given on Page 58 of \cite{chris}. The cited figure is used to prove the equality of $r_I$ and $l_I$ for monoidal categories. All inner diagrams can be ﬁlled by weakening 2-morphisms to obtain the result. The figure above, regardless of considering $\epsilon$, commutes, which is shown by Gurski in Proposition 4.24 \cite{Gurski}. 
\subsection{Tensorators}\label{sec:tensorators}
Kapranov and Voevodsky enumerated three types of tensorators as data because he did not mention that the tensor product $\otimes$ is a 2-functor at heart. They are detonated as $\otimes_{f,g}$, $\otimes_{f^\prime,f,B}$ and $\otimes_{A,g^\prime,g}$ and defined in Figure \ref{kv_tensors}.
\begin{figure}[h]
	\centering
\begin{tikzpicture}
\node (0) at (0, 0) {$A^\prime \otimes B$};
\node (1) at (4, 0){$A^\prime \otimes B^\prime$};
\node (2) at (0, 2){$A \otimes B$};
\node (3) at (4, 2){$A \otimes B^\prime$};
\node[rotate= -135] (4) at (2, 1){$\Rightarrow$}; 
\node (5) at (2.3, 1.3){$\otimes_{f, g}$};
\draw[->] (2)  to node[below, sloped]{$f \otimes B$} (0);
\draw[->] (2) to node[above]{$A \otimes g$} (3);
\draw[->] (3) to node[above, sloped]{$f \otimes B^\prime$} (1);
\draw[->] (0) to node[below]{$A^\prime \otimes g$}(1);
\end{tikzpicture}
\\
\begin{tikzpicture}
\node (0) at (0, 0){$A \otimes B$}; 
\node (1) at (0, 3){$A^\prime \otimes B$}; 
\node (2) at (3, 3){$A^{\prime \prime} \otimes B$}; 
\node[rotate=-45] (3) at (1, 2){$\Rightarrow$}; 
\node (4) at (1, 2.5){$\otimes_{f', f, B}$};

\draw[->] (0) to node[above, sloped]{$f\otimes B$}(1);
\draw[->](1) to node[above]{$f'\otimes B$}(2);
\draw[->] (0) to node[below, sloped]{$f'f\otimes B$} (2); 
\end{tikzpicture}
\begin{tikzpicture}
	\node (0) at (0, 0){$A \otimes B$}; 
	\node (1) at (0, 3){$A \otimes B^\prime$}; 
	\node (2) at (3, 3){$A \otimes B^{\prime \prime}$}; 
	\node[rotate=-45] (3) at (1, 2){$\Rightarrow$}; 
	\node (4) at (1, 2.5){$\otimes_{A, g^\prime, g}$};
	
	\draw[->] (0) to node[above, sloped]{$A\otimes g$}(1);
	\draw[->](1) to node[above]{$A\otimes g^\prime$}(2);
	\draw[->] (0) to node[below, sloped]{$A\otimes g^\prime g$} (2); 
\end{tikzpicture}
\caption{KV's tensorators}\label{kv_tensors}
\end{figure}

We can translate them based on our tensorator $\phi$. For the first case, $\otimes_{f, g} =\phi^{-1}_{A^\prime, f, g, B} \odot \phi_{f, B^\prime, A, g} $. 
$$
(f \otimes B^\prime)\circ (A \otimes g) \xRightarrow{\phi_{f, B^\prime, A, g}} (f \circ A) \otimes (B^\prime \circ g ) = (A^\prime \circ f) \otimes (g \circ B)  \xRightarrow{\phi^{-1}_{A^\prime, f, g, B}} (A^\prime \otimes g )\circ (f \otimes B)$$
The second and third KV's tensorators are labelled by two 1-morphisms and an object, based on $\phi$, they are $\otimes_{f', f, B} = \phi_{f', B, f, B}$ and $\otimes_{A, g', g} = \phi_{A, g', A, g}$.  
\subsection{Consequence of $\phi_{A, B} = 1$}\label{sec:consequence}
1-Associators and 1-unitors apart from Axiom \ref{eq:transformation1} should further satisfy Axiom \ref{eq:transformation2}, which seemingly we have ignored throughout the paper. This axiom is built on the assumption that $id_{A \otimes B}$ and $id_A \otimes id_B$ are isomorphic not equivalent, $id_{A\otimes B} \overset{\phi_{A, B}}{\cong} id_A \otimes id_B$. However, to obtain Kapranov and Voevodsky’s definition, here we assume $\phi_{A,B}$ are identities, which results in the proceeding corollaries.
\begin{figure}
	\centering
	\begin{tikzpicture}
	\node (0) at (0, 0) {$(AB)C$};
	\node (1) at (4, 0) {$A(BC)$};
	\node (2) at (0, 4) {$(AB)C$};
	\node (3) at (4, 4) {$A(BC)$};
	\node[red, rotate=135] (4) at (1.5, 2.5){$\Rightarrow$}; 
	\node[red] (5) at (1.5, 2.9){$l_{a_{A, B, C}}$}; 
	\node[blue, rotate=135] (6) at (1.5, 2){$\Rightarrow$}; 
	\node[blue] (7) at (1.5, 1.5){$\alpha_{id_{A, B, C}}$};
	\node (8) at (4.6, 2){$\Rightarrow$}; 
	\node (8) at (4.6, 2.3){$\gamma$}; 
	\node (8) at (-0.8, 2){$\Rightarrow$}; 
	\node (8) at (-0.8, 2.3){$\beta$};

	\draw[->] (0) to node[below, sloped]{$a_{A, B, C}$}(1);
	\draw[->, blue] (0) to node[above, sloped]{$id_A \otimes (id_B \otimes id_C)$} (2);
	\draw[->, red] (1) to node[above, sloped]{$id_{A(BC)}$} (3);
	\draw[->] (2) to node[above, sloped]{$a_{A, B, C}$}(3);
	\draw[->, bend left=60, red] (0) to node[above, sloped]{$id_{(AB)C}$}(2);
	\draw[->, bend right = 60, blue] (1) to node[below, sloped]{$id_A \otimes (id_B \otimes id_C)$}(3);
\end{tikzpicture}\label{fig:naturality}
\caption{Figure of naturality section}
\end{figure}
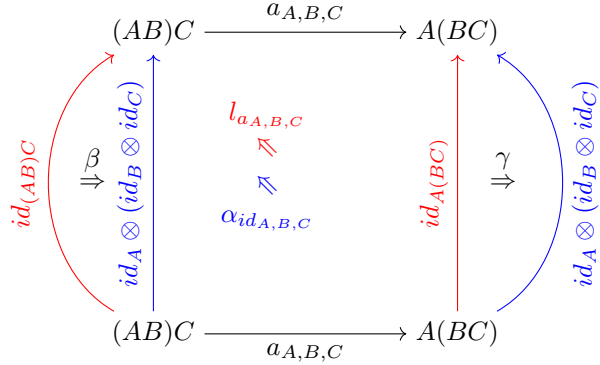
\begin{cro}
In a monoidal 2-category with the definition given above, 2-unitors are identity 2-morphisms if they are indexed by identity 1-morphisms.
\end{cro}
\begin{proof}
In Equation \ref{eq:transformation2}, let $F = - \otimes I$ and $G = Id$, $\sigma_A = r_A$ and $\sigma_f = r^{-1}_f$. We have $r_{id_A} = (1 * \phi_{A, I})$, which means $r_{id_A} = 1$.
\end{proof}
\begin{cro}
In a monoidal 2-category with the definition given above, 2-associators are identities if at least one of their indices is identity 1-morphism.
\end{cro}
\begin{proof}
Check Figure \ref{fig:naturality}, if  $\beta = \phi_{AB, C} \odot (\phi_{A, B} \otimes C)$ and $\gamma = \phi_{A, BC} \odot (A \otimes \phi_{B, C})$. 
The blue square is naturality and the filling 2-associator is $(a_{A, B, C} \circ \beta^{-1}) \odot (\gamma \circ a_{A, B, C}) = \alpha_{id_A, B, C}$. If one lets $\phi_{A,B}$ identity, then 2-associators become identities if at least one of their indices is identity. 

\end{proof}
\begin{rem}
Note that since we are working with 2-categories not bicategories, we easily let $a_{A,B,C} \circ id_{(AB)C} = id_{A(BC)} \circ a_{A,B,C} = a_{A,B,C}$.
\end{rem}
\bibliographystyle{Apalike}

\newpage
\appendix 
\section{Figures}\label{sec:figures}
\subsection{Diagrams of Naturality}
\subsubsection{Naturality of 1-Associators}\label{sec:nat-1-ass}
\noindent
\begin{itemize}
\item \underline{Naturality of 1-associators in 1-morphism $f$}
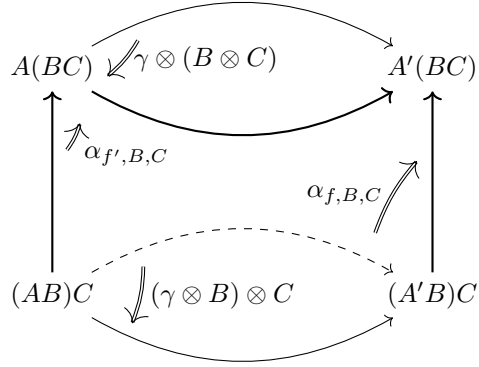
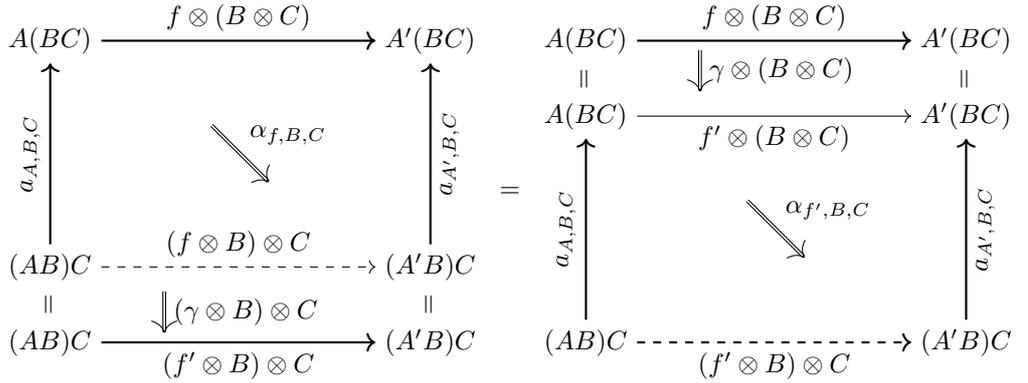
\begin{equation}\label{fig:nat-ass1}
	\input{images/alpha-natural-3d.tikz}
\end{equation}
\input{images/alpha-natural-cyl.tikz}
\item \underline{Naturality of 1-associators in Object $B$}
the missing 2-morphisms in the blue squares are $\beta = 
\phi^{-1}_{A', g\otimes C, f, B\otimes C}\odot \phi^{-1}_{g, c, B, C} \odot (f \otimes \phi_{B', c, g, C}) \odot \phi_{f, B'\otimes C, A, g\otimes C}
$ and $\gamma = (\phi_{A', g, f, B}^{-1} \odot \phi_{f, B', A, g}) \hat{\otimes} C$. 
\begin{equation}
\input{images/naturalityinB.tikz}
\end{equation}
\begin{equation*}
\input{images/naturalityinB2.tikz}
\end{equation*}
\begin{equation}\label{fig:nat-1-ass2}
	\input{images/naturalityinB3.tikz}
\end{equation}
\newpage
\item \underline{Axiom \ref{eq:transformation1} of naturality:} let $F(A, B, C) = (AB)C$, $G(A, B, C) = A(BC)$, $\sigma_{A,B,C} = a_{A,B,C}$, $\sigma_f = \alpha_{f,B,C}$ and $\kappa = \psi = \phi$, it is straightforward to ﬁrst obtain the planar diagrams, then from shared boundaries glue them to obtain the 3-dimensional diagram of Kapranov and Voevodsky in Page 219. The missing 2-morphisms in the blue triangles, $\beta = \phi_{f', BC, f, BC}$ and $\gamma = \phi_{f', B, f, B} \hat{\otimes} C$.
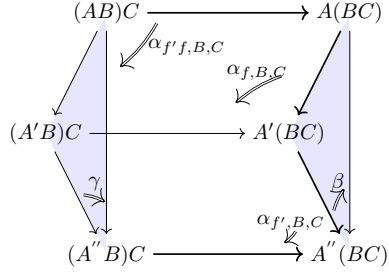
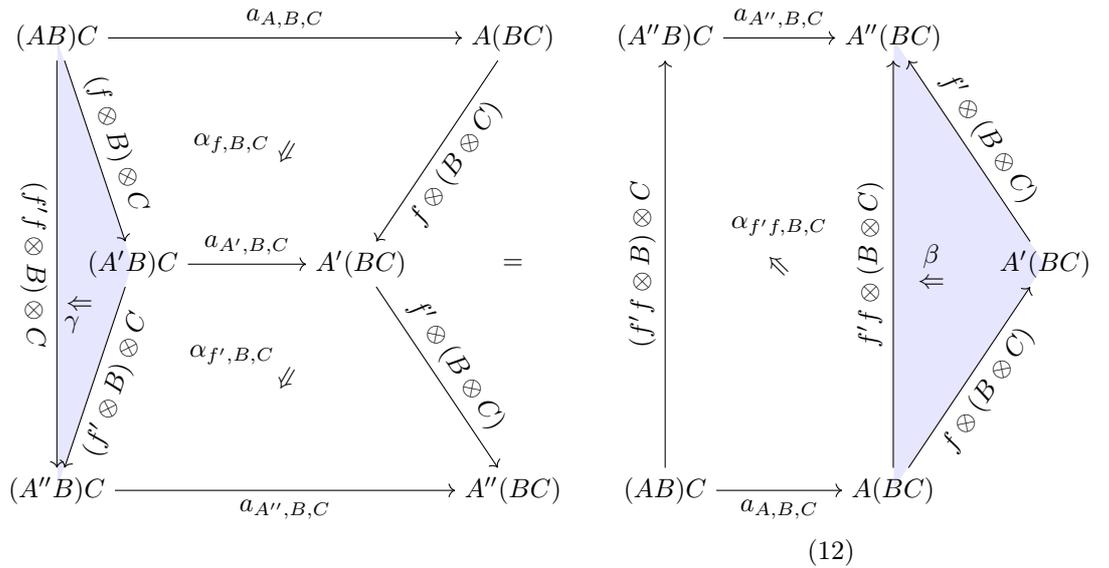
\begin{equation*}
\input{images/compat.tikz} 
\end{equation*}
\begin{equation}\label{fig:axiom-ass}
\input{images/compat11.tikz}
\end{equation}
\end{itemize}
\subsubsection{Naturality of 1-Unitors}
\begin{itemize}
\item \underline{Naturality of 1-unitors in 1-morphism $f$:}
\begin{equation*}
\input{images/naturalityofl.tikz}
\end{equation*}
\begin{equation}\label{fig:nat-unit1}
\input{images/naturalityofl2.tikz}
\end{equation}
\item \underline{Axiom \ref{eq:transformation1} for unitors} let $GA = A$, $FA = IA$, $\sigma_A = l_A$, $\sigma_f = l_f$, $\kappa =1$, $\psi = \phi$.
\begin{equation*}
	\input{images/compat3.tikz}
\end{equation*}
\begin{equation}\label{fig:nat-unit2}
	\input{images/triangle-prism.tikz}
\end{equation}
\end{itemize}
\subsubsection{Naturality of Tensorators}
\begin{itemize}
\item \underline{Naturality of tensorator 1} The missing 2-morphism is $\beta = \phi^{-1}_{A', g, f, B} \odot \phi_{f, B', A, g}$, $\eta = \phi^{-1}_{A', g, f', B} \odot \phi_{f', B', A, g}$. 
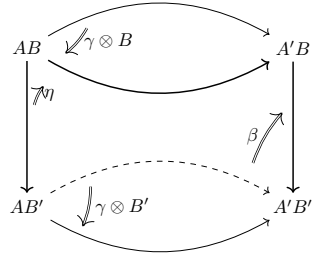
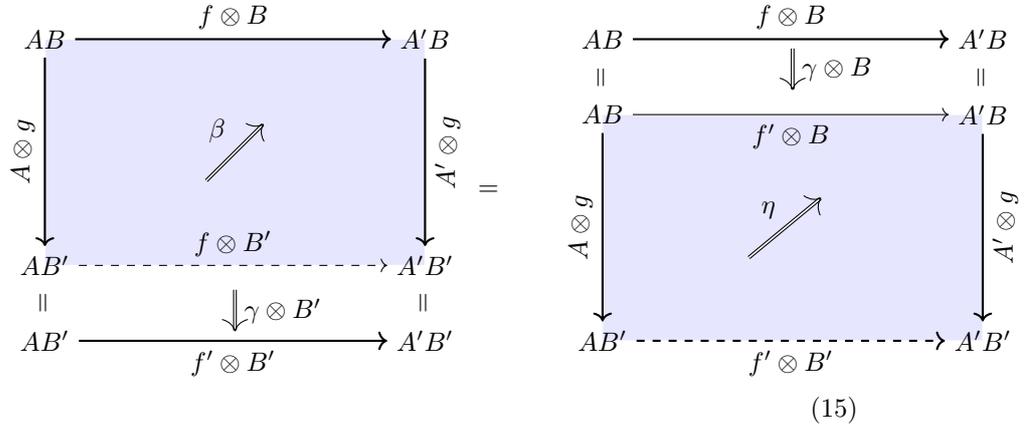
\begin{equation*}
\input{images/naturalityoftens.tikz}
\end{equation*}
\begin{equation}\label{fig:nat-tensor1}
\input{images/naturalityoftens11.tikz}
\end{equation}
\item \underline{Naturality of tensorator 2:}
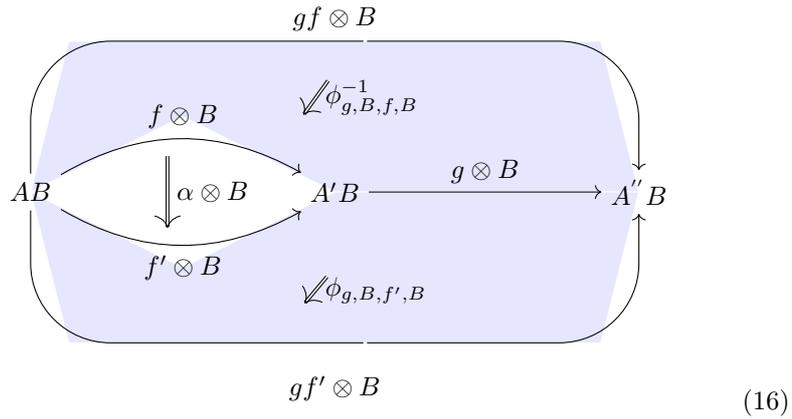
\begin{equation}\label{fig:nat-tensor2}
\input{images/naturalityoftens2.tikz}
\end{equation}
\item \underline{Naturality of tensorator 3:} the missing 2-morphisms in the figure \ref{fig:nat-tensor3} are 
$
\xi = \phi_{f', B, f, B}, \beta = \phi_{f', B', f, B'}, \gamma = \phi^{-1}_{A'', g, f'f, B} \odot \phi_{f'f, B', A, g}, \eta= \phi^{-1}_{A'', g, f', B} \odot \phi_{f', B', A', g}, \lambda = \phi^{-1}_{A', g, f, B}\odot \phi_{f, B', A, g}$.

\begin{equation*}
\input{images/naturalityoftens3.tikz}
\end{equation*}
\begin{equation}\label{fig:nat-tensor3}
\input{images/natualityoftens31.tikz}
\end{equation}
\item \underline{Axiom \ref{eq:transformation1} of naturality} Let $F(A) = AB$. The filling 2-morphisms are $\lambda = \phi_{f^\prime,B,f,B}$, $\beta = \phi_{f^{\prime\prime},B,f^\prime f,B}$,  $\gamma = \phi_{f^{\prime\prime}f^\prime,B,f,B}$, $\xi = \phi_{f^{\prime\prime},B,f^\prime,B}$. 
\begin{equation*}
	\input{images/compat2.tikz}
\end{equation*}
\begin{equation}\label{fig:axiom-tensor}
	\input{images/compat21.tikz}
\end{equation}
\end{itemize}
\subsection{Diagrams of Modification}
\subsubsection{2-Unitors}
\underline{Modification condition for 2-unitors}: We only describe the condition on $\eta$, other ones are similar. The missing 2-morphism is $\gamma = \phi^{-1}_{A^\prime, l_{B}, f, IB} \odot \phi_{f, B, A, l_B}$. 
\begin{equation*}
\input{images/naturalityof2-unit1.tikz} 
\end{equation*}
\begin{equation}\label{fig:modi-unit1}
\input{images/naturalityof2-unit11.tikz}
\end{equation}
\subsubsection{Pentagonator}
Modification condition for pentagonator: the 2-isomorphism filling the front square of 3D picture is
$
\gamma = \phi^{-1}_{A^\prime, a_{B, C, D}, f, ((BC)D)} \odot 
\phi_{f, B(CD), A, a_{B, C, D}}
$
\begin{equation*}
\input{images/pi-natural-1.tikz}
\end{equation*}
\begin{equation}\label{fig:modi-penta}
	\input{images/pi-natural-2.tikz}	
\end{equation}
\input{images/pi-natural-3.tikz}
\subsection{Axioms}
\subsubsection{Stasheff Polytope}
\underline{Stasheff Polytope for pentagonator:} To draw this diagram, we used the coordinates of Stasheff Polytope in Stay's paper \cite{Stay}.
\begin{equation}\label{fig:stashef}
\input{images/stasheff.tikz}
\end{equation}
\input{images/stasheff2.tikz}
\subsubsection{Unit Polytopes}
\begin{itemize}
\item \underline{Unit polytope I}
\begin{equation*}
\input{images/polytope.tikz}
\end{equation*}
\begin{equation}\label{fig:unit-poly1}
\input{images/polytope2.1.tikz}
\end{equation}
\input{images/polytope2.2.tikz}
\item \underline{Unit polytope II}
\begin{equation}\label{fig:unit-poly2}
\input{images/poly.tikz}
\end{equation}
\end{itemize}
\section{Definition of 2-Functor}\label{sec:2-functor}
Leinster \cite{Leinster} presented a definition for bifunctor, when source and target are bicategories rather than 2-categories. Since we work with 2-categories, we modify the definition as presented below. We illustrate similar diagrams as Leinster and only add $=$ instead of isomorphism. 
\begin{defi}
A 2-functor $F: \mathcal{A} \longrightarrow \mathcal{B}$ consists of the following data which satisfy the two axioms given below. 
\begin{itemize}
	\item function $F:obj(\mathcal{A}) \longrightarrow \mathcal{B}$
	\item functors $F_{AB}: \mathcal{A}(A, B) \longrightarrow \mathcal{B}(A, B)$
	\item natural transformations $\phi_{g, f}: Fg \circ Ff \longrightarrow F(g\circ f)$, $\phi_A: id_{FA} \longrightarrow F(id_A)$
	\item \textbf{Axiom 1}: 
	\begin{equation}\label{tensorator:axiom1}
		\input{images/2functor1.tikz}
	\end{equation}
	\item \textbf{Axiom 2}: 
\begin{equation}\label{tensorator:axiom2}
	\input{images/2functor2.tikz}
\end{equation}
\end{itemize}
\begin{defi}
A transformation $\sigma: F\Rightarrow G$ where $(F, \kappa)$ and $(G, \psi)$ are morphisms, with data 1-morphism $\sigma_A: FA \longrightarrow GA$ and natural 2-isomorphism $\sigma_f : Gf \circ  \sigma_A  \Rightarrow \sigma_B \circ  Ff$, such that
\begin{equation}\label{eq:transformation1}
	\input{images/2functor3.tikz}
\end{equation}
and 
\begin{equation}\label{eq:transformation2}
	\input{images/2functor4.tikz}
\end{equation}
\end{defi}
\end{defi}
\section{Modification Data}\label{sec:modification-data}
\begin{defi}
A morphism between two natural morphisms $(\sigma, \tilde{\sigma})$ between two 2-functors is called modification $\Gamma$. For every object $A$, it consists of a 2-morphism $\Gamma_A$, subjects to the modification square \ref{modification-square} for every morphism $f$. 
\begin{equation}\label{modification-square}
\input{images/modification1.tikz}
\end{equation}
\end{defi}
\subsection{Details of Modification Diagram for 2-Unitors}
We only describe the details of our procedure for obtaining the prism for one of 2-unitors, namely, $\mu$, the recipe works for other 2-unitors as well. In square \ref{modification-square}, let $G(A, B) = (AI)B$ and $F(A, B) = AB$, $\sigma_{A, B} = (A \otimes l_B) a_{A, I, B}$ and $\tilde{\sigma}_{A, B} = r_A \otimes B$. Now, we find the edges of Diagram \ref{unitor-modification-square} and compose the 2-morphisms. 
\begin{equation}\label{unitor-modification-square}
	\input{images/unitor-modi1.tikz}
\end{equation}
\begin{itemize}
\item{\textbf{($\tilde{\sigma}_f \odot  (1_{Gf}* \mu_{A, B})$)}}\\
	\input{images/unitor-modi2.tikz}
\item{\textbf{$(\mu_{A, B}*1_{Ff})\odot \sigma_f$}}\\
	\input{images/unitor-modi3.tikz} \\
	\input{images/unitor-modi4.tikz}
\end{itemize}
\subsection{Details of Modification Diagram for Pentagonator}
To unpack the modification square, let 2-functors be $G(A, B, C, D) = ((AB)C)D$,
and $F(A, B, C, D) = A(B(CD))$, natural transformations are $\sigma_{A,B,C,D} = (A \otimes
a_{B,C,D})  a_{A,BC,D}  (a_{A,B,C} \odot D)$ and $\tilde{\sigma}_{A,B,C,D} = a_{A,B,CD} a_{AB,C,D}$.
\begin{equation}
\input{images/pent-modifi1.tikz}
\end{equation}
\begin{itemize}
	\item{\textbf{($\tilde{\sigma}_f \odot  (1_{Gf}* \mu_{A, B})$)}}\\
	\input{images/pent-modifi3.tikz} 
	\item{\textbf{$(\mu_{A, B}*1_{Ff})\odot \sigma_f$}}\\
	\input{images/pent-modifi4.tikz}
\end{itemize}
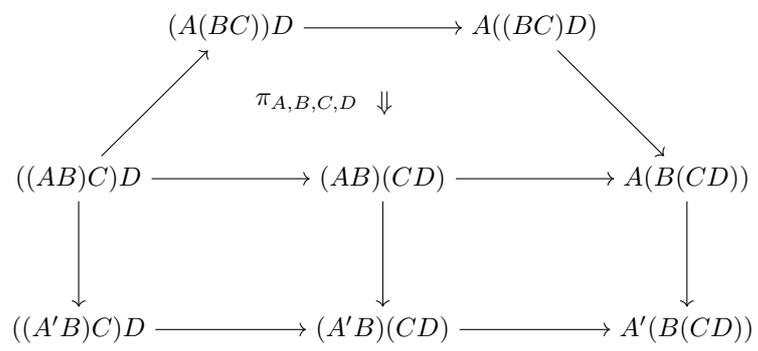
\end{document}

%% file: images/2functor2.tikz
\begin{tikzpicture}
\node (0) at (0, 0) {$Ff \circ id_{FA}$};
\node (1) at (3, 0) {$Ff\circ Fid_A$};
\node (2) at (6, 0) {$F(f \circ id_A)$};
\node (3) at (3, -2){$Ff$}; 
\draw[->] (0) to node[above]{$1*\phi_A$}(1); 
\draw[->] (1) to node[above]{$\phi$}(2);
\draw[->] (0) to node[below]{$r_F$}(3) ;
\draw[->] (2) to node[below]{$=$}(3); 
\end{tikzpicture}